\documentclass{amsart}
\usepackage{amsmath, amssymb, mathrsfs, verbatim, multirow}

\usepackage[all]{xy}
\usepackage{pifont}
\usepackage{float}
\usepackage{color}
\usepackage{enumitem}
\newcommand{\bigslant}[2]{{\raisebox{.2em}{$#1$}\left/\raisebox{-.2em}{$#2$}\right.}}

\usepackage{tikz}
\usepackage{tikz-cd}
\usepackage{tkz-tab}
\usepackage{subcaption}

\usepackage[hidelinks]{hyperref}

\newtheorem{Teo}{Theorem}[section]
\newtheorem{Prop}[Teo]{Proposition}
\newtheorem{Lema}[Teo]{Lemma}
\newtheorem{Cor}[Teo]{Corollary}

\theoremstyle{theorem}
\newtheorem{Def}[Teo]{Definition}

\newtheorem{Obs}[Teo]{Remark}

\newcommand{\Q}{\mathbb{Q}}

\newcommand{\Z}{\mathbb{Z}}
\newcommand{\N}{\mathbb{N}}

\newcommand{\lra}{\longrightarrow}

\newcommand{\VR}{\mathcal{O}}

\newcommand{\MI}{\mathfrak{m}}

\newcommand{\supp}{\mbox{\rm supp}}

\newcommand{\K}{\mathbb{K}}

\DeclareMathOperator{\inv}{in}
\begin{document}
\title{Graded rings associated to valuations and direct limits}
\author{Silva de Souza, C. H., Novacoski, J. A. and Spivakovsky, M.}
\thanks{During the realization of this project the authors were supported by a grant from Funda\c c\~ao de Amparo \`a Pesquisa do Estado de S\~ao Paulo (process numbers 2017/17835-9 and 2021/13531-0).}

\begin{abstract} In this paper, we study the structure of the graded ring associated to a limit key polynomial $Q_n$ in terms of the key polynomials that define $Q_n$. In order to do that, we use direct limits. In general, we describe the direct limit of a family of graded rings associated to a totally ordered set of valuations. As an example, we describe the graded ring associated to a valuation-algebraic valuation as a direct limit of graded rings associated to residue-transcendental valuations. 

\end{abstract}


\maketitle

\section{Introduction}

The graded ring structure associated to a valuation $\nu$,  denoted by $\mathcal{G}_\nu$ (see \linebreak Definition~\ref{defAnelGrad}), has proved to be an important object  on Valuation Theory. For \linebreak example, the graded ring describes information of the value group $\nu\K$ and the residue field $\K\nu$ 
 simultaneously. It was proved in \cite{Matheus} that $\mathcal{G}_\nu$ is isomorphic to the semigroup ring $\K\nu[t^{\nu \K}]$ with a suitable multiplication. Also, this structure is related to the program developed by Teissier to prove \textit{local uniformization}, an open problem in positive characteristic  with  applications in resolution of singularities. This program is based on the study of the spectrum of certain graded rings (see \cite{teissier}). 

\vspace{0.2cm}

Other important objects, which are also linked with  programs to prove local uniformization,  are key polynomials (see Definition~\ref{defiPoliChave}). These polynomials 
were introduced by  Mac Lane in  \cite{MacLane} and  generalized years later by Vaquié in \cite{Vaq}, using the structure of graded ring. We will refer to them as \textit{Mac Lane-Vaquié key polynomials}.  
After that, Novacoski and Spivakovsky in \cite{josneiKeyPolyPropriedades} and Decaup, Mahboub
and Spivakovsky in \cite{spivamahboubkeypoly} introduced a new notion of key polynomials, which is the one we use in this paper. These two definitions can be well understood by using graded rings, as one can see in \cite{Andrei} and \cite{josneimonomial}. 

\vspace{0.2cm}

Among key polynomials,  the so called limit key polynomials are of great interest to us. Limit key polynomials were introduced in \cite{Vaq} and are one of the main aspects of the generalization of Mac Lane's original key polynomials by Vaquié. Here we use a formulation similar to the one presented in \cite{josneiKeyPolyPropriedades} (see Definition~\ref{defLKP2}). These polynomials  are related to the existence of defect, which is an obstacle when dealing with valuations and local uniformization. For example, in the case  where the valuation has a unique extension, the defect is the product of the \textit{relative degrees} of  limit key polynomials (see  \cite{Nart},  \cite{Saturnino} or \cite{Vaq2}). 

\vspace{0.2cm}

%

%

%
%
%

For a valuation $\nu$ on $\K[x]$, we consider  the set $\Psi_n$ of all key polynomials for $\nu$ of degree $n$.  In this paper, we study the structure of the graded ring associated to a limit key polynomial $Q_n$ for $\Psi_n$, denoted by $\mathcal{G}_{Q_n}$,  in terms of the key polynomials $Q\in \Psi_n$.  A  related problem was studied in \cite{Vaq}. Here, we approach this problem by describing $\mathcal{G}_{Q_n}$ as the direct limit of a direct system defined by the graded rings $\mathcal{G}_Q$ and the maps presented in Section~\ref{GradedRings}.

\vspace{0.2cm}

Take a valuation $\nu_0$  on $\K$ with value group $\Gamma_0$. Fix a totally ordered divisible group $\Gamma$ containing $\Gamma_0$.  Take $\Gamma_\infty:=\Gamma \cup\{\infty\}$ with the usual extension of  addition and  order. Let
$$\mathcal{V} = \{\nu_0\}\cup \{\nu: \K[x] \rightarrow \Gamma_\infty \mid \nu \text{ is a valuation extending } \nu_0  \} .$$
 Consider the partial order on $\mathcal{V}$ given by $\nu_0\leq \nu$ for every $\nu\in \mathcal{V}$ and, for $\nu, \mu\in \mathcal{V}\setminus\{   \nu_0\}$, we set $\nu \leq \mu$ if and only if $\nu(f)\leq \mu(f)$ for every $f\in \K[x]$. Our first result deals with an arbitrary totally ordered subset $\mathfrak{v}= \{\nu_i\}_{i\in I}\subset \mathcal{V}$ such that there exists  $\nu\in \mathcal{V}$ satisfying $\nu_i\leq \nu$ for every $i\in I$. Theorem~\ref{teoLimVmaiorRstable} will give us that $\underset{\longrightarrow}{\lim\,}\mathcal{G}_{\nu_i}$
is isomorphic to the additive subgroup $R$ of $\mathcal{G}_{\nu}$ generated by the set
  $ \{ \inv_\nu(f)\mid f \text{ is } \mathfrak{v}\text{-stable} \} $ ($\mathfrak{v}$-stability is defined in Section~\ref{TotallyOrderLimits}).

\vspace{0.2cm}

Next, we divide the totally ordered subsets $\mathfrak{v}= \{\nu_i\}_{i\in I}\subset \mathcal{V}$ into three types: the ones with maximum, the ones without maximum such that every $f\in \K[x]$ is $\mathfrak{v}$-stable and the ones  without maximum such that there exists at least one polynomial that is not $\mathfrak{v}$-stable. We show that in the first and second cases there exists $\nu \in \mathcal{V}$, that we will denote by $\displaystyle\sup_{i\in I}\nu_i $, satisfying $\nu\geq \nu_i$ for every $i\in I$ and  $R = \mathcal{G}_{\nu}$ (Proposition~\ref{lemVstable} and Corollary~\ref{teoLimitSeqStable}). In the third case, we show that for a polynomial $Q$ of smallest degree that is not $\mathfrak{v}$-stable we can define $\mu\in \mathcal{V}$ such  that $\mu$ is equal to its truncation at $Q$ (see Definition~\ref{defTruncamento}), $\mu\geq \nu_i$ for every $i\in I$ and $R = R_Q$, where \linebreak $R_Q= \langle \{ \inv_\mu(f)\mid \deg(f)<\deg(Q) \}\rangle\subset \mathcal{G}_\mu$ (Proposition~\ref{lemamuQgammaVal} and Corollary~\ref{teoLimitSeqNotStable}).

%

\vspace{0.2cm}

We then give two applications of the previous results. The first one concerns limit key polynomials, our main interest. We prove that, given a limit key polynomial $Q_n$ for $\Psi_n$, the subset $\mathfrak{v}= \{\nu_Q\}_{Q\in \Psi_n}\subset \mathcal{V}$ is totally ordered without a maximum and $Q_n$ is a polynomial of smallest degree that is not $\mathfrak{v}$-stable (Corollary~\ref{corQnNotvStable}). Therefore,  $\underset{\longrightarrow}{\lim\,}\mathcal{G}_{\nu_i}\cong R_{Q_n}$ (Corollary~\ref{corLimQn}).

\vspace{0.2cm}

The second application concerns valuation-algebraic valuations (see Definition~\ref{defValAlg}). We prove that, given a valuation-algebraic valuation $\nu$, there exists a totally ordered subset $\mathfrak{v}= \{\nu_Q\}_{Q\in \textbf{Q}}\subset \mathcal{V}$   without maximum with each $Q$ a key polynomial for $\nu$, $\nu_Q$ a residue-transcendental valuation and $\nu = \displaystyle\sup_{Q\in \textbf{Q}}\nu_Q$ (Proposition~\ref{propValAlgStable}).  Therefore,  $\underset{\longrightarrow}{\lim\,}\mathcal{G}_{Q}\cong \mathcal{G}_{\nu}$ (Corollary~\ref{corValgLim}).

\vspace{0.2cm}
 
%
%
%
%

This paper is  organized as follows. In Section \ref{Preli}, we present the main definitions and results that will be used throughout the paper.  In Section \ref{GradedRings}, we present the main results about graded rings associated to a valuation that will be useful in our discussions. In Section \ref{TotallyOrderLimits}, for  a given totally ordered subset $\mathfrak{v}= \{\nu_i\}_{i\in I}\subset \mathcal{V}$, we begin presenting  some properties of the direct limit of the direct system $\{( \mathcal{G}_{\nu_i}, \phi_{ij})\}^{i,j\in I }_{i\leq j}$ and prove Theorem~\ref{teoLimVmaiorRstable}. Then we prove Corollary~\ref{teoLimitSeqStable} and Corollary~\ref{teoLimitSeqNotStable} in  Subsections \ref{FirstSecondCases} and \ref{ThirdCase}, respectively. In Section \ref{LKPLimit},  we \linebreak describe the graded ring associated to a limit key  polynomial via Corollary~\ref{corLimQn}. In Section \ref{ValAlgLimit}, we describe the graded ring associated to a valuation-algebraic valuation via Corollary~\ref{corValgLim}.

\par\medskip
\textbf{Acknowledgements.} We would like to thank the anonymous referee for carefully reading, providing useful suggestions and pointing out a few mistakes in an earlier version of this paper.

\section{Preliminaries}\label{Preli}


\begin{Def}
Take a commutative ring $R$ with unity. A \index{Valuation}\textbf{valuation} on $R$ is a mapping $\nu:R\lra \Gamma_\infty :=\Gamma \cup\{\infty\}$ where $\Gamma$ is a totally ordered abelian group (and the extension of addition and order to $\infty$ is done in the natural way), with the following properties:
\begin{description}
	\item[(V1)] $\nu(ab)=\nu(a)+\nu(b)$ for all $a,b\in R$.
	\item[(V2)] $\nu(a+b)\geq \min\{\nu(a),\nu(b)\}$ for all $a,b\in R$.
	\item[(V3)] $\nu(1)=0$ and $\nu(0)=\infty$.
\end{description}
\end{Def}

\vspace{0.2cm} 	

 Let $\nu: R \lra\Gamma_\infty$ be a valuation. The set $\supp(\nu)=\{a\in R\mid \nu(a )=\infty\}$ is called the \textbf{support of $\nu$}. The \textbf{value group of $\nu$} is the subgroup of $\Gamma$ generated by 
$\{\nu(a)\mid a \in R\setminus \supp(\nu) \}
$ and is denoted by $\nu R$. A valuation $\nu$ is a \index{Valuation!Krull}\textbf{Krull valuation} if $\supp(\nu)=\{0\}$.  If $\nu$ is a Krull valuation, then $R$ is a domain and we can extend $\nu$ to $\K={\rm Quot}(R)$ on the usual way. In this case, 
 define the \textbf{valuation ring} as $\VR_\nu:=\{ a\in\K\mid \nu(a)\geq 0 \}$. The ring $\VR_\nu$ is a local ring with unique maximal ideal $\MI_\nu:=\{a\in\K\mid \nu(a)>0 \}. $ We define the \textbf{residue field} of $\nu$ to be the field $\VR_\nu/\MI_\nu$ and denote it by $\K\nu$. The image of $a\in \VR_\nu$ in $\K\nu$ is denoted by $a\nu$. 
 
 \vspace{0.2cm} 	

\begin{Obs}Take  a valuation $\nu$ on a field $\K$ and  a valuation  $\overline{\nu}$ on $\overline{\K}$, the algebraic closure of $\K$, such that $\overline{\nu}|_{\K}=\nu$. Then $\overline{\nu}\overline{\K}$ is a divisible group. Additionally, $\overline{\nu}\overline{\K}=\nu\K\otimes_\Z \Q$ (see \cite{Eng}, p.79). About the residue fields, it is known that $\overline{\K}\overline{\nu} $ is the algebraic closure of $\K\nu$ (see \cite{Eng}, p.66).

\end{Obs}

\vspace{0.2cm} 	


Fix a valuation  $\nu$ on $\K[x]$, the ring of polynomials in one indeterminate over the field $\K$. Our main definition of key polynomial relates to the one in \cite{josneiKeyPolyPropriedades}, 
 which is related to the one in    \cite{spivamahboub}.  Fix an algebraic closure $\overline{\K}$  for $\K$ and fix  a valuation $\overline{\nu}$ on $\overline{\K}[x]$ such that $\overline{\nu}|_{\K[x]}=\nu$.

\vspace{0.2cm} 	

\begin{Def}
	
	Let $f\in \overline{\K}[x]$ be a non-zero polynomial.
	
	\begin{itemize}
	\item If $\deg(f)>0$, set 
	$$\delta(f):=\max\{\overline{\nu}(x-a)\mid a\in \overline{\K} \text{ and }f(a)=0\}.$$
	
	\item If $\deg(f)=0$, set $\delta(f) = -\infty$.
	
	\end{itemize}
%
%

\end{Def}

\vspace{0.2cm} 	

\begin{Obs} According to \cite{josneiKeyPolyMinimalPairs}, $\delta(f)$ does not depend on the choice of the algebraic closure $\overline{\K}$ or the extension $\overline{\nu}$ of $\nu$. 

\end{Obs}

\vspace{0.2cm}

\begin{Def}\label{defiPoliChave}
	A monic polynomial $Q\in \K[x]$ is a  \textbf{key polynomial} of level $\delta(Q)$ if, for every $f\in \K[x]$,
	$$\delta(f)\geq \delta(Q) \Longrightarrow \deg(f)\geq \deg(Q). $$
\end{Def}

%
\vspace{0.2cm} 	

 Let $q\in \K[x]$ be a non-constant polynomial. Then there exist uniquely determined polynomials $f_0,\ldots, f_s\in \K[x]$ with $\deg(f_i)<\deg(q)$ for every $i$, $0\leq i\leq s$, such that
 \begin{equation}\label{eqQexpansion}
 f = f_0+f_1q+\ldots +f_sq^s.
 \end{equation}
We call this expression the \textbf{\textit{q}-expansion} of $f$.

\vspace{0.2cm} 	

\begin{Prop}\label{propequivPolichave}(Lemma 2.3 \textbf{(iii)} of \cite{josneiKeyPolyPropriedades}  + Corollary 3.52 of \cite{leloup}) 
 Let $Q\in \K[x]$ be a monic polynomial. The following assertions are equivalent.

\begin{enumerate}[label=\rm (\roman*):]

\item $Q$ is a key polynomial for $\nu$.

\item For every $f,g\in \K[x]$ with $\deg(f)<\deg(Q)$ and $\deg(g)<\deg(Q)$, if \linebreak $fg=lQ+r$ is the $Q$-expansion of $fg$, then $\nu(fg)=\nu(r)<\nu(lQ)$.
\end{enumerate}
%
%
%
	
\end{Prop}

\vspace{0.2cm}

\begin{Def}\label{defTruncamento} Let $q\in \K[x]$ be a non-constant polynomial and $\nu$ a valuation on $\K[x]$. 
The map 
$$\nu_q(f):=\underset{0\leq i\leq s}{\min} \{ \nu(f_iq^i)  \},$$
for $f\in \K[x]$ as in (\ref{eqQexpansion}), is called the \textbf{truncation} of $\nu$ at $q$.
\end{Def}
\vspace{0.2cm} 	

%
 This map is not always a valuation, as we can see in Example 2.4 of \cite{josneiKeyPolyPropriedades}.  

\vspace{0.2cm}

\begin{Prop}(Proposition 2.6 of \cite{josneiKeyPolyMinimalPairs}) If $Q$ is a key polynomial, then $\nu_Q$ is a valuation on $\K[x]$.
\end{Prop}

\vspace{0.2cm} 	

In the next lemmas, we state some properties of key polynomials and truncations. In what follows, we denote by $\Gamma_\Q:=\nu(\K[x])\otimes_\Z\Q$ the divisible hull of $\nu(\K[x])$.

\vspace{0.2cm} 

\begin{Lema}\label{prop3itensQQlinhaPoliChaves}(Proposition 2.10 of \cite{josneiKeyPolyPropriedades})
	Let $Q,Q'\in \K[x]$  be key polynomials for $\nu$.
	We have the following.
	\begin{enumerate}
		\item If $\deg(Q)<\deg(Q')$, then $\delta(Q)<\delta(Q')$.
		
		\item If $\delta(Q)<\delta(Q')$, then  $\nu_Q(Q')<\nu(Q')$.
		
		\item If $\deg(Q)=\deg(Q')$, then
		$$\nu(Q)<\nu(Q')\Longleftrightarrow \nu_Q(Q')<\nu(Q')\Longleftrightarrow \delta(Q)<\delta(Q'). $$
	\end{enumerate}
\end{Lema}

\vspace{0.2cm} 

\begin{Lema}\label{lempolichaveEpsilonmenor} (Corollaries 3.9, 3.10, 3.11 and 3.13 of \cite{josneimonomial}) Let $Q,Q'\in \K[x]$  be key polynomials such that $\delta(Q)\leq \delta(Q')$.
\begin{enumerate}

\item We have $\nu_{Q'}(Q)=\nu(Q)$.

\item  For every $f\in \K[x]$, we have
$\nu_Q(f)\leq \nu_{Q'}(f). $
In particular, if $\delta(Q)= \delta(Q')$ then $\nu_Q=\nu_{Q'}$.

\item For every $f\in \K[x]$, if $\nu_Q(f)=\nu(f)$, then $\nu_{Q'}(f)=\nu(f)$.

\item  For every $f\in \K[x]$, if $\delta(Q)< \delta(Q')$ and $\nu_{Q'}(f)<\nu(f)$,  then \linebreak $\nu_Q(f)< \nu_{Q'}(f). $
\end{enumerate}

\end{Lema}

\vspace{0.2cm} 

\begin{Obs}In \cite{josneiKeyPolyPropriedades} and \cite{josneimonomial} the definition of key polynomial and the above results are stated using the notion of $\epsilon(f)$ instead of $\delta(f)$. For a non-zero polynomial  $f\not \in \supp(\nu)$ with $\deg(f)>0$, one defines
		$$\epsilon(f):=\underset{1\leq b\leq \deg(f)}{\max}\left\lbrace\left.\frac{\nu(f) - \nu(\partial_bf)}{b} \;\right| \partial_bf\not\in \supp(\nu) \right\rbrace\in \Gamma_\Q, $$ where $\partial_bf$ is the formal  Hasse-derivative of order $b$ of $f$. If $f\not \in \supp(\nu)$ and $\deg(f)=0$, then we set  $\epsilon(f):=-\infty$ and if $f\in\supp(\nu)$, then we set $\epsilon(f):= \infty$.
  Proposition 3.1 of \cite{josneiKeyPolyMinimalPairs} shows that $\delta(f)=\epsilon(f)$ for all $f\in \K[x]$.

\end{Obs}

\vspace{0.2cm} 

\section{Graded ring associated to a valuation}\label{GradedRings}

Let $\nu$  be a valuation on $\K[x]$.
For each $\gamma\in \nu(\K[x])$, we consider the abelian groups 
$$\mathcal{P}_\gamma = \{ f\in \K[x]\mid \nu(f)\geq \gamma \} \text{ and }  \mathcal{P}_{\gamma}^{+} = \{ f\in \K[x]\mid \nu(f)>\gamma \}. $$

\vspace{0.2cm} 

\begin{Def}\label{defAnelGrad}
The \textbf{graded ring } associated to $\nu$  is defined by 
	$$\mathcal G_\nu={\rm gr}_{\nu}(\K[x]):= \bigoplus_{\gamma\in \nu(\K[x])}\mathcal{P}_\gamma/ \mathcal{P}_{\gamma}^{+}.$$	
\end{Def}

\vspace{0.2cm} 

The sum on $\mathcal G_\nu$ is defined coordinatewise and the product is given by extending  the product of homogeneous elements, which is described by 
$$\left( f+ \mathcal{P}_{\beta}^{+}  \right) \cdot \left( g+ \mathcal{P}_{\gamma}^{+} \right): = \left(  fg +\mathcal{P}_{\beta+\gamma}^{+}\right), $$
where $\beta= \nu(f)$ and $\gamma = \nu(g)$.

\vspace{0.2cm} 

For $f\not \in \supp(\nu)$,  we denote by $\inv_\nu(f)$ the image of  $f$ in $\mathcal{P}_{\nu(f)}/ \mathcal{P}_{\nu(f)}^{+} \subseteq \mathcal{G}_\nu.$ \linebreak If $f\in \supp(\nu)$, then we define $\inv_\nu(f)=0$. The next lemma follows from the definitions above. 

\vspace{0.2cm} 

\begin{Lema}\label{lemPropinv} 
Let $f,g\in \K[x]$. We have the following. 
	\begin{enumerate}
		
		\item $\mathcal{G}_\nu$ is an integral domain. 
		
		\item $\inv_\nu(f)\cdot \inv_\nu(g)=\inv_\nu(fg)$.

		\item $\inv_\nu(f)=\inv_\nu(g)$ if and only if 
		$\nu(f-g)>\nu(f)=\nu(g)$.
	\end{enumerate}

\end{Lema}

\vspace{0.2cm} 

Let $\nu_i$ and $\nu_j$ be valuations on $\K[x]$ such that $\nu_i(f)\leq \nu_j(f)$ for all $f\in \K[x]$. Let $\mathcal{P}_\gamma(\K[x],\nu_i) =\{ f\in \K[x]\mid \nu_i(f)\geq \gamma \} $ (analogously we define $\mathcal{P}_\gamma(\K[x],\nu_{j}), \mathcal{P}_{\gamma}^{+}(\K[x],\nu_i)$ and $\mathcal{P}_{\gamma}^{+}(\K[x],\nu_{j})$).
We have the inclusions
	$$\mathcal{P}_\gamma(\K[x],\nu_i)\subseteq  \mathcal{P}_\gamma(\K[x],\nu_{j})$$
	and
	$$\mathcal{P}_{\gamma}^{+}(\K[x],\nu_i)\subseteq  \mathcal{P}_{\gamma}^{+}(\K[x],\nu_{j}) $$
	for any $\gamma\in \nu_i(\K[x])\subseteq \nu_{j}(\K[x])$.
We consider the following map: 
\begin{align}\label{eqPhiij}
	\phi_{ij}: \hspace{0.3cm}\mathcal G_{\nu_i}\hspace{0.2cm} &\longrightarrow \hspace{0.2cm} \mathcal G_{\nu_j}\\
	\inv_{\nu_i}(f) &\longmapsto  \begin{cases}
		\inv_{\nu_j}(f)& \mbox{ if }\nu_i(f)=\nu_{j}(f)\\ 
		0&\mbox{ if }\nu_i(f)<\nu_{j}(f),
	\end{cases}  \nonumber
	\end{align}

\noindent and we extend this map naturally for an arbitrary element. This map is well-defined (Corollary 5.5 of \cite{Andrei}) and, by construction, it is a homomorphism of graded rings.

\vspace{0.2cm} 

Suppose that  $q\in \K[x]$ is such that $\nu_q$ is a valuation.
Let $R_q$  be the additive subgroup of $\mathcal G_{\nu_q}$ generated by the set
$$\{ \inv_{\nu_q}(f)\mid f\in \K[x]_d  \}, $$
where  $d=\deg(q)$ and $\K[x]_d=\{ f\in \K[x]\mid \deg(f)<d \}$. We set $y_q:=\inv_{\nu_q}(q)$. The next propositions say that a non-zero $y_q$ can be seen as a transcendental element over $R_q$ and that $R_q$ is a subring of $\mathcal G_{\nu_q}$ if and only if $q$ is a key polynomial for $\nu$. 

\begin{Prop}\label{lemAlgGraduadaAnelPoliy}(Proposition 4.5 of \cite{josneimonomial}) We have 
$$ \mathcal G_{\nu_q}= R_q[y_q].$$
Moreover, if $q\in  \supp(\nu)$, then $\mathcal G_{\nu_q} = R_q$. If $q\not\in \supp(\nu)$, then $y_q$ is such that if 
$$a_0+a_1y_q+\ldots +a_sy_q^s=0 $$
for some $a_0, \ldots, a_s \in R_q$, then $a_i =0$ for every $i$, $0\leq i \leq s$. 
\end{Prop}

\vspace{0.2cm}

%

\begin{Prop}\label{propequivPolichaveRq}(Theorem 5.7 of \cite{josneicaio})
Suppose $\nu_q$ is a valuation on $\K[x]$. Then the following assertions are equivalent.

\begin{enumerate}[label=\rm (\roman*):]

\item $q$ is a key polynomial for $\nu$.

\item The set $R_q$ is a subring of $\mathcal G_{\nu_q}$.

\end{enumerate}
%
%
%
%
%
\end{Prop}

\section{Totally ordered sets of valuations and direct limits}\label{TotallyOrderLimits}

Take $\mathcal{V}$ as in the introduction, that is,
%
$$\mathcal{V} = \{\nu_0\}\cup \{\nu: \K[x] \rightarrow \Gamma_\infty \mid \nu \text{ is a valuation extending } \nu_0  \}. $$
Consider the partial order on $\mathcal{V}$ given by $\nu_0\leq \nu$ for every $\nu\in \mathcal{V}$ and, for $\nu, \mu\in \mathcal{V}\setminus\{   \nu_0\}$, we set $\nu \leq \mu$ if and only if $\nu(f)\leq \mu(f)$ for every $f\in \K[x]$. We say that $\nu < \mu$ if $\nu\leq \mu$ and there exists $g\in \K[x]$ such that $\nu(g)<\mu(g)$. In $\mathcal{V}$ we have the following property.

\vspace{0.2cm} 

\begin{Prop}\label{prop3valMenor} (Proposition 2.2 of \cite{novbarnabe}) Assume $\eta, \nu, \mu\in \mathcal{V}$ are such that \linebreak  $\eta<\nu<\mu$. For $f\in \K[x]$, if $\eta(f)=\nu(f)$, then $\nu(f)=\mu(f)$.
\end{Prop}

\vspace{0.2cm}

Let $\mathfrak{v}= \{\nu_i\}_{i\in I}\subset \mathcal{V}$ be a totally ordered set. Consider a total order on the index set  $I$ induced from the order on $\mathfrak{v}$. 
 Since we have a total order, 
$(I, \leq)$ is a directed set.\footnote{That is, $\leq$ is reflexive and transitive relation on $\mathfrak{v}$ such that, for every $\nu_i,\nu_j\in \mathfrak{v}$, there exists $\nu_k\in \mathfrak{v}$ satisfying  $\nu_i\leq \nu_k$ and $\nu_j\leq \nu_k$.}
%
%
%

\vspace{0.2cm}

\begin{Lema}\label{lemDirectSystemValTotalOrder}
	Let $\mathfrak{v}= \{\nu_i\}_{i\in I}$ be a totally ordered set in $\mathcal{V}$. Consider the family of graded rings $\{  \mathcal{G}_{\nu_i}\}_{i\in I }$.
	For $\nu_i \leq  \nu_j$, let $\phi_{ij}$ be the map
	\begin{align*}
	\phi_{ij}:\hspace{0.3cm}\mathcal G_{\nu_i}\hspace{0.2cm} &\longrightarrow \hspace{0.2cm}\mathcal G_{\nu_j}\\
	\inv_{\nu_i}(f) &\longmapsto  \begin{cases}
		\inv_{\nu_j}(f)& \mbox{ if }\nu_i(f)=\nu_{j}(f)\\ 
		0&\mbox{ if }\nu_i(f)<\nu_{j}(f),
	\end{cases} 
	\end{align*}
extended in a natural way to arbitrary (that is, not necessarily homogeneous) elements of $\mathcal G_{\nu_i}$. Then  $\{( \mathcal{G}_{\nu_i}, \phi_{ij})\}^{i,j\in I }_{i\leq j}$ is a direct system over $I$.

\end{Lema}

\begin{proof}We need to check  that $\phi_{ii}$ is the identity map and $\phi_{ik}=\phi_{jk}\circ \phi_{ij}$ for all $i\leq j \leq k$.
\begin{itemize}
\item By definition, $\phi_{ii}(\inv_{\nu_i}(f)) = \inv_{\nu_i}(f)$ for every $f\in \K[x]$, hence $\phi_{ii}$ is the identity  map on $\mathcal{G}_{\nu_i}$.

 \item Take $i\leq j \leq k$,
  that is,  
 \begin{equation}\label{eqnuinujnuk}
 \nu_i(f)\leq \nu_{j}(f)\leq \nu_{k}(f)
 \end{equation} for  all $f\in \K[x]$.
If  the strict inequality holds in some of the inequalities of \eqref{eqnuinujnuk}, then 
$$(\phi_{jk}\circ \phi_{ij}) (\inv_{\nu_i}(f)) = 0 = \phi_{ik}(\inv_{\nu_i}(f)). $$
If $\nu_i(f)=\nu_{j}(f)= \nu_{k}(f)$, then
 $$(\phi_{jk}\circ \phi_{ij}) (\inv_{\nu_i}(f)) = \inv_{\nu_k}(f) = \phi_{ik}(\inv_{\nu_i}(f)). $$
%
%
%
Hence, $\phi_{ik}=\phi_{jk}\circ \phi_{ij}$ for all $i\leq j \leq k$.
\end{itemize}

Therefore, 
   $\{( \mathcal{G}_{\nu_i}, \phi_{ij})\}^{i,j\in I }_{i\leq j}$ is a direct system over $I$.

\end{proof}

\begin{Obs} We do not gain in generality if we suppose $\mathfrak{v}= \{\nu_i\}_{i\in I}$ simply a directed set, because every directed set in $\mathcal{V}$ is totally ordered. Indeed, if $\mathfrak{v} \subset \mathcal{V}$ is a directed set, then given $\nu, \mu\in \mathfrak{v}$ there exists $\eta\in \mathfrak{v}$ such that $\nu\leq \eta$ and $\mu \leq \eta$. By Theorem 2.4 of \cite{Nart}, the set $(-\infty, \eta) = \{\rho\in \mathcal{V}\mid \rho< \eta \}$ is totally ordered.
 Therefore, $\nu$ and $\mu$ are comparable.

\end{Obs}

\vspace{0.2cm} 

We want to describe the direct limit of the direct system $\{( \mathcal{G}_{\nu_i}, \phi_{ij})\}^{i,j\in I }_{i\leq j}$. 
We are going to use the characterization  of $\underset{\longrightarrow}{\lim\,}\mathcal{G}_{\nu_i}$ as a disjoint union.
The direct limit of the direct system $\{( \mathcal{G}_{\nu_i}, \phi_{ij})\}^{i,j\in I }_{i\leq j}$ is defined as
$$\underset{\longrightarrow}{\lim\,}\mathcal{G}_{\nu_i}:= \bigslant{\displaystyle\bigsqcup_{i\in I}\mathcal{G}_{\nu_i}}{\sim}, $$
where $\sim$ is the following equivalence relation: for $a_i\in \mathcal{G}_{\nu_i}$ and $a_j\in \mathcal{G}_{\nu_j}$ with $i\leq j$, 
\begin{equation*}\label{relacaoequiv}
a_i\sim a_j \Longleftrightarrow  \phi_{ij}(a_i)=a_j.
\end{equation*}

We denote by $[a_i]$ the equivalence class of $a_i$ in $\underset{\longrightarrow}{\lim\,}\mathcal{G}_{\nu_i}$. The operations on $\underset{\longrightarrow}{\lim\,}\mathcal{G}_{\nu_i}$ are induced from each $\mathcal{G}_{\nu_i}$.
 Denoting by $0_i$ the additive identity of $\mathcal{G}_{\nu_i}$, it is easy to see that $[0_i]=[0_j]$ for all $i,j\in I$. We write only $[0]$ to denote $[0_i]$, which is the additive identity of $\underset{\longrightarrow}{\lim\,}\mathcal{G}_{\nu_i}$. Similarly, if $\inv_{\nu_i}(1)$  is the  multiplicative identity of $\mathcal{G}_{\nu_i}$, it is easy to see that $[\inv_{\nu_i}(1)]=[\inv_{\nu_j}(1)]$ for all $i,j\in I$. The equivalence class $[\inv_{\nu_i}(1)]$ is the multiplicative identity of $\underset{\longrightarrow}{\lim\,}\mathcal{G}_{\nu_i}$.  We also have the following.
 
 \vspace{0.2cm} 

\begin{Lema}\label{lemClassLimite}For   fixed $k\in I$ and $f\in \K[x]$, consider $[\inv_{\nu_k}(f)]\in \underset{\longrightarrow}{\lim\,}\mathcal{G}_{\nu_i}$. 
\begin{enumerate}
\item  We have $\nu_k(f)<\nu_j(f)$ for some $j>k$ if and only if $[\inv_{\nu_k}(f)]=[0]$.

\item  For $j\geq k$, if $\nu_k(f)=\nu_j(f)$, then $[\inv_{\nu_k}(f)]=[\inv_{\nu_j}(f)]$. Moreover, if $\nu_k(f)=\nu_j(f)$ for every $j\geq k$, then $[\inv_{\nu_k}(f)]\neq [0]$.
\end{enumerate}

\end{Lema}

\begin{proof}$\,$

\begin{enumerate}

\item If $\nu_k(f)<\nu_j(f)$ for some $j>k$, then 
$$\phi_{kj}(\inv_{\nu_k}(f)) = 0_j. $$
Hence, $[\inv_{\nu_k}(f)]=[0]$. On the other hand, if $[\inv_{\nu_k}(f)]=[0]$, then there exists  $j\geq k$ such that $$\phi_{kj}(\inv_{\nu_k}(f)) =0_j.$$ By the definition of $\phi_{kj}$, this implies $\nu_k(f)<\nu_j(f)$ and $k<j$. 

\item  For $j\geq k$, if $\nu_k(f)=\nu_j(f)$, then
$$\phi_{kj}(\inv_{\nu_k}(f)) =  \inv_{\nu_j}(f). $$
Hence, we have $[\inv_{\nu_k}(f)]=[\inv_{\nu_j}(f)]$. Moreover, if $\nu_k(f)=\nu_j(f)$ for every $j\geq k$, then by the preceding item we have $[\inv_{\nu_k}(f)]\neq [0]$.

\end{enumerate}
\end{proof}

Let $\mathfrak{v}= \{\nu_i\}_{i\in I}$ be a totally ordered subset of $\mathcal{V}$. For every $f\in \K[x]$, we say that $f$ is $\boldsymbol{\mathfrak{v}}$\textbf{-stable} if there exists $i_f\in I$ such that 
\begin{equation}\label{eqVStable}
\nu_i(f)=\nu_{i_f}(f) \text{ for every } i\in I \text{ with } i\geq i_f.
\end{equation}

In the following, let $\mathfrak{v}= \{\nu_i\}_{i\in I}$ be a totally ordered subset of $\mathcal{V}$ and suppose there exists $\nu\in \mathcal{V}$ such that $\nu_i\leq \nu$ for every $i\in I$.
\vspace{0.2cm} 

\begin{Lema}\label{lemMenorIgualMu}  Fix $k\in I$. We have the following. 
\begin{enumerate}

\item If $\nu_k(f)<\nu(f)$, then $\nu_k(f)<\nu_j(f)$ for every $j>k$.

\vspace{0.1cm} 

\item We have $\nu_k(f) = \nu(f)$  if and only if $\nu_k(f)=\nu_j(f)$ for every $j\geq k$. In other words, $f$ is $\mathfrak{v}$-stable if and only if 
 $\nu_k(f)=\nu(f)$ for some $k\in I$.
\end{enumerate}
\end{Lema}

\begin{proof}$\,$
\begin{enumerate}

\item  If there exists $j> k$ such that $\nu_k(f)=\nu_j(f)$, then  $\nu_k(f)=\nu_j(f)<\nu(f)$ and $\nu_k<\nu_j<\nu$, contradicting Proposition~\ref{prop3valMenor}.

\vspace{0.1cm} 

\item Suppose $\nu_k(f)=\nu(f)$. If  there exists $j>k$ such that $\nu_k(f)<\nu_j(f)$, then  $\nu(f)=\nu_k(f)<\nu_j(f)$, which is a contradiction.  Conversely, suppose $\nu_k(f)=\nu_j(f)$ for every $j\geq k$. If $\nu_k(f)<\nu(f)$, then by the preceding item we would have $\nu_k(f)<\nu_j(f)$ for every $j>k$, contradicting our assumption. 

\end{enumerate}
\end{proof}

 Let $S\subset  \mathcal{G}_{\nu} $ be any subset. We denote by $\langle S \rangle$ the additive subgroup of  $\mathcal{G}_{\nu}$  generated by $S$. We define
$$R = \langle \{ \inv_\nu(f)\mid f \text{ is } \mathfrak{v}\text{-stable} \} \rangle\subseteq  \mathcal{G}_{\nu}. $$

\begin{Lema}\label{lemR}  We have the following. 
\begin{enumerate}

\item $R$ is a subring of $\mathcal{G}_{\nu}$.

\vspace{0.1cm} 

\item Let $d\in \N\cup\{\infty\}$ be the smallest positive integer for which some polynomial of degree is not $\mathfrak{v}$-stable (if there is no such polynomial, set $d=\infty$). Then 
$$R = \langle \{ \inv_\nu(f)\mid f \in \K[x]_d \} \rangle.$$

\end{enumerate}
\end{Lema}

\begin{proof}$\,$
\begin{enumerate}

\item  By construction, $R$ is an additive subgroup of $\mathcal{G}_{\nu}$ and clearly $\inv_\nu(1)\in R$. For $\inv_\nu(f), \inv_\nu(g)\in R$,  we take $k = \max\{  i_f, i_g\}$. Then, for every $j\geq k$, it follows that
$$\nu_j(fg) = \nu_j(f)+\nu_j(g) = \nu_k(f)+\nu_k(g) = \nu_k(fg). $$
That is, $fg$ is $\mathfrak{v}$-stable and then $\inv_\nu(fg)\in R$. This shows that $R$ is a subring of $\mathcal{G}_{\nu}$. 

\vspace{0.1cm} 

\item If $d=\infty$, then $K[x]_d=K[x]$ and every $f$ is $\mathfrak{v}$-stable. Hence, the result follows.  
Suppose $d<\infty$ and take $Q\in \K[x]$ not $\mathfrak{v}$-stable and $\deg(Q)=d$. We denote $R'=\langle \{ \inv_\nu(f)\mid f \in \K[x]_d \} \rangle.$
 We have $R'\subseteq R$ since $\deg(f)<d$ implies $f$ is $\mathfrak{v}$-stable (by the minimality of $d=\deg(Q)$).

Consider $\inv_\nu(f)\in R$, hence $f$ is $\mathfrak{v}$-stable. Let $f = gQ+f_0$ the euclidean division of $f$ by $Q$. Thus, $\deg(f_0)<d$ and $f_0$ is also $\mathfrak{v}$-stable. By Lemma~\ref{lemMenorIgualMu} and the fact that $Q$ is not $\mathfrak{v}$-stable, we have that $\nu(Q)>\nu_i(Q)$ for every $i\in I$.
Take $k = \max\{i_f, i_{f_0}\}$. Hence, 
$\nu_k(f) = \nu_j(f) \text{ and }  \nu_k(f_0) = \nu_j(f_0) $
for every $j\geq k$. 
 By Lemma~\ref{lemMenorIgualMu},  
$\nu(f) =  \nu_k(f) $ and $\nu(f_0)=\nu_k(f_0)$. 
Thus,
\begin{align*}
\nu(f-f_0)&=\nu(gQ)\\
&>\nu_k(gQ)\\
& \geq \min\{ \nu_k(f), \nu_k(f_0)\}\\
& =\min\{ \nu(f), \nu(f_0)\}.
\end{align*}
%
Hence, $\nu(f-f_0)>\nu(f) = \nu(f_0)$, that is, $\inv_{\nu}(f) = \inv_{\nu}(f_0)\in R'$. Therefore, $R=R'$.
\end{enumerate}
\end{proof}

Now we present our main result. 

\vspace{0.2cm}

\begin{Teo}\label{teoLimVmaiorRstable}
We have $$\underset{\longrightarrow}{\lim\,}\mathcal{G}_{\nu_i}\cong R.$$

\end{Teo}

\begin{proof}

 Consider the map given by
\begin{align*}
\tau: \hspace{0.5cm} R\hspace{0.3cm} &\longrightarrow \underset{\longrightarrow}{\lim\,}\mathcal{G}_{\nu_i}\\
\inv_{\nu}(f) & \longmapsto [\inv_{\nu_{i_f}}(f)],
\end{align*}
where we extend $\tau$ naturally to arbitrary (that is, not necessarily homogeneous) elements of $R$ by additivity. This map is well defined. Indeed,
  take $\inv_\nu(f)\in R$. By assumption, $f$ is $\mathfrak{v}$-stable, so there exists $i_f\in I$ such that \eqref{eqVStable} is satisfied. If $j_f$ is another index such that $\nu_i(f)=\nu_{j_f}(f)$ for every $i\in I$ with $i\geq j_f$, then without loss of generality we can take $i_f\leq j_f$.  Hence, by Lemma~\ref{lemClassLimite} (2), $[\inv_{\nu_{i_f}}(f)] =[\inv_{\nu_{j_f}}(f)]$.
  
  \vspace{0.1cm} 

Also, suppose $f,g\in \K[x]$ are such that $\inv_\nu(f)=\inv_\nu(g)\in R$. By Lemma~\ref{lemPropinv} (3), this means that $\nu(f-g)>\nu(f)=\nu(g)$. By Lemma~\ref{lemR} (2), 
we can consider $f,g\in \K[x]_d$. Hence, $f-g\in \K[x]_d$ and $f,g$ and $f-g$ are $\mathfrak{v}$-stable. Take $k=\max\{i_f, i_g, i_{f-g}\}$. By Lemma~\ref{lemMenorIgualMu}, $$\nu_k(f)=\nu(f),\,\nu_k(g)=\nu(g) \text{ and } \nu_k(f-g)=\nu(f-g). $$
Then, 
 \begin{align*}\nu_k(f-g)&=\nu(f-g)\\
 & > \min\{\nu(f), \nu(g) \} \\
 & = \min\{\nu_k(f), \nu_k(g) \}.  
 \end{align*}
We conclude that $\inv_{\nu_k}(f)=\inv_{\nu_k}(g)$. Therefore, using Lemma~\ref{lemClassLimite} we have
$$\tau(\inv_\nu(f)) = [\inv_{\nu_{i_f}}(f) ] =  [\inv_{\nu_{k}}(f) ]=[\inv_{\nu_{k}}(g) ] = [\inv_{\nu_{i_g}}(g) ]=\tau(\inv_\nu(g))$$
and we see that
 $\tau$ is well defined.

  Since we extended $\tau$ to  arbitrary elements of $R$ via finite sums, this map is a group homomorphism by construction. We now check that $\tau$ is a ring isomorphism. 

\begin{itemize}
\item $\tau$ is injective: since it is a group homomorphism, it is enough to check that $\ker(\tau)=\{0\}$.  Given a non-zero element $\inv_\nu(f)\in R$, we know that $\nu_{i_f}(f)= \nu_j(f)$ for every $j\geq i_f$. By Lemma~\ref{lemClassLimite} (2),  $\tau(\inv_\nu(f)) = [\inv_{\nu_{i_f}}(f)] \neq  [0]$. Hence, $\ker(\tau)=\{0\}$ and  $\tau$ is injective.

\vspace{0.1cm}

\item $\tau$ is surjective: take any $[\inv_{\nu_k}(f)]\in \underset{\longrightarrow}{\lim\,}\mathcal{G}_{\nu_i}$. If there exists $j>k$ such that $\nu_k(f)<\nu_j(f)$, then by Lemma~\ref{lemClassLimite} (1) we have $[\inv_{\nu_{i_f}}(f)] = [0]=\tau(0)$. On the other hand, if $\nu_k(f) =\nu_j(f)$ for every $j\geq k$, then we can take $i_f=k$ and $[\inv_{\nu_k}(f)] = [\inv_{\nu_{i_f}}(f)] = \tau(\inv_\nu(f))$. Therefore, $\tau$ is surjective. 

\vspace{0.1cm} 

\item $\tau$ is a ring homomorphism: for any $\inv_\nu(f), \inv_\nu(g)\in R$, we can take $j\in I$ sufficiently large so that $j\geq \max\{i_{fg}, i_f, i_g\}$. We have
\begin{align*}
\tau(\inv_\nu(f) \cdot \inv_\nu(g) ) & = \tau(\inv_\nu(fg))\\
									 & = [\inv_{\nu_{j}}(fg)]\\
									 & = [\inv_{\nu_{j}}(f)\cdot \inv_{\nu_{j}}(g)]\\
									 & = [\inv_{\nu_{j}}(f)]\cdot [\inv_{\nu_{j}}(g)]\\					
									 & = \tau(\inv_\nu(f))\cdot \tau(\inv_\nu(g)).
\end{align*}
Also, $\tau$ preserves the multiplicative identity since, by definition, $\tau(\inv_\nu(1)) = [\inv_{\nu_i}(1)]$ (for any $i\in I$), which is the unity of $\underset{\longrightarrow}{\lim}\mathcal{G}_{\nu_i}$.
%
%
%
\end{itemize}

Therefore, we have $\underset{\longrightarrow}{\lim\,}\mathcal{G}_{\nu_i}\cong R$  
as commutative rings with unity.
  
\end{proof}

We will classify the totally ordered subsets $ \mathfrak{v}\subset \mathcal{V}$ in three classes using the following proposition.
%
%
%
%
%
%
%
%



\vspace{0.2cm}

\begin{Prop} (Corollary 2.3 of \cite{novbarnabe}) Let $\{\nu_i\}_{i\in I}$ be a totally ordered set in $\mathcal{V}$. For every $f\in \K[x]$, either $\{\nu_i(f)\}_{i\in I}$ is strictly increasing, or there exists $i_0\in I$ such that $\nu_i(f)=\nu_{i_0}(f)$ for every $i\in I$ with $i\geq i_0$.  
\end{Prop}

\vspace{0.2cm} 

We consider three cases:

\begin{itemize}
\item $\mathfrak{v}$ has maximum $\nu_m$.

\item $\mathfrak{v}$ has no maximum and every $f\in \K[x]$ is $\mathfrak{v}$-stable.

\item $\mathfrak{v}$ has no maximum and there exists $q\in \K[x]$  not $\mathfrak{v}$-stable.
\end{itemize}


\subsection{First and second cases}\label{FirstSecondCases}

\begin{Prop}\label{lemVstable}Let $\mathfrak{v}= \{\nu_i\}_{i\in I}$ be a totally ordered set in $\mathcal{V}$
 such that 
  every $f\in \K[x]$ is $\mathfrak{v}$-stable. 
Define $\nu = \displaystyle\sup_{i\in I}\nu_i: \K[x] \rightarrow \Gamma_\infty$ by $\nu(f) = \nu_{i_f}(f)$.  
Then $\nu$ is a valuation on $\K[x]$ such that $\nu_i\leq \nu$. Moreover, if $\nu'\in \mathcal{V}$ is such that $\nu'\leq \nu$ and $\nu_i\leq \nu'$ for every $i\in I$, then  $\nu'=\nu$. 

\end{Prop}

\begin{proof}
Take $f,g\in \K[x]$. Then, by assumption, there exist $i_f,\, i_g,\, i_{f+g}$ and $i_{fg}$ satisfying (\ref{eqVStable}). Take $j=\max\{i_f,\, i_g,\, i_{f+g}, \,i_{fg}\}$. Hence, $$\nu(fg)=\nu_{j}(fg)=\nu_{j}(f)+ \nu_{j}(g) = \nu(f)+\nu(g)$$
and
$$\nu(f+g) = \nu_{j}(f+g)\geq \min\{ \nu_{j}(f), \nu_{j}(g)\}=\min\{ \nu(f), \nu(g)\}. $$
Also, $\nu(0)=\nu_{i_0}(0)=\infty$ and $\nu(1)=\nu_{i_1}(1)=0$. Therefore, $\nu$ is a valuation on $\K[x]$. In addition, for each $f\in \K[x]$ and $i\in I$, we have $\nu_i(f)\leq \nu_{i_f}(f)=\nu(f)$. Hence, $\nu_i\leq \nu$. 

Moreover, suppose $\nu'\in  \mathcal{V}$ is such that $\nu'\leq \nu$ and $\nu_i\leq \nu'$ for every $i\in I$. Thus, for every $f\in \K[x]$ we have 
$$\nu(f)\geq \nu'(f)\geq \nu_{i_f}(f)=\nu(f). $$
Therefore, $\nu'=\nu$.
%
%

\end{proof}



\begin{Obs}If $\mathfrak{v}= \{\nu_i\}_{i\in I}$ has maximum $\nu_{m}$, then every $f$ is $\mathfrak{v}$-stable (take $i_f=m$). Hence, $\nu$ in Proposition~\ref{lemVstable} coincides with $\nu_{m}$. 
\end{Obs}

\vspace{0.2cm} 

We have the following corollary, which covers the first and second cases.

\vspace{0.2cm} 

\begin{Cor}\label{teoLimitSeqStable} Let $\mathfrak{v}= \{\nu_i\}_{i\in I}$ be a totally ordered set in $\mathcal{V}$ such that 
every $f\in \K[x]$ is $\mathfrak{v}$-stable.  Consider the direct system $\{( \mathcal{G}_{\nu_i}, \phi_{ij})\}^{i,j\in I }_{i\leq j}$. Take $\nu = \displaystyle\sup_{i\in I}\nu_i$  as in Proposition~\ref{lemVstable}. Then $\underset{\longrightarrow}{\lim\,}\mathcal{G}_{\nu_i}\cong \mathcal{G}_{\nu}$  
as commutative rings with unity.

\end{Cor}

\begin{proof}  It follows from Theorem~\ref{teoLimVmaiorRstable} because $R = \mathcal{G}_{\nu}$.

\end{proof}

\subsection{Third case}\label{ThirdCase}

Now we treat the third case.  Let $\mathfrak{v}= \{\nu_i\}_{i\in I}$ be a totally ordered set in $\mathcal{V}$ such that $\mathfrak{v}$ has no maximum and there is at least one polynomial that is not $\mathfrak{v}$-stable. 
Consider the set 
$$C(\mathfrak{v}):=\{ f\in \K[x]\mid f \text{ is } \mathfrak{v}\text{-stable }\}.$$
For every $f\in C(\mathfrak{v})$ we set $\mathfrak{v}(f) = \nu_{i_f}(f)$. Let $Q$ be a monic polynomial of smallest degree $d$ not $\mathfrak{v}$-stable and take $\gamma\in \Gamma_\infty$ such that $\gamma>\nu_i(Q)$ for every $i\in I$. 

\vspace{0.2cm} 
Consider the map 
$$\mu(f_0+f_1Q+\ldots+f_rQ^r)=\underset{0\leq j \leq r}{\min}\{ \mathfrak{v}(f_j)+j\gamma  \}, $$
where $f_0+f_1Q+\ldots+f_rQ^r$ is the $Q$-expansion of $f$.

\begin{Prop}\label{lemamuQgammaVal} We have the following.
\begin{enumerate}
\item  We have $\mu\in \mathcal{V}$.

\item We have $\nu_i<\mu$ for every $i\in I$.

\item We have $\mu = \mu_Q$ and $Q$ is a key polynomial for $\mu$. 
\end{enumerate}

\end{Prop}

\begin{proof}$\,$

\begin{enumerate}
\item By Theorem 2.4 of \cite{novbarnabe}, $\mu$ is valuation. It follows from the definition that $\nu_0=\mu\mid_\K$.

%

\item  Using Proposition 1.21 of \cite{Vaq} and Theorem 5.1 of \cite{MacLane}, one can prove that $ \nu_i(f)\leq \mu(f)$ for every $f\in \K[x]$. Also,  $\nu_i(Q)<\gamma=\mu(Q)$ for every $i\in I$. 
Hence, $\nu_i<\mu$ for every $i\in I$.

\item It follows immediately from the definition of $\mu$ that  $\mu = \mu_Q$. 
We now prove that $Q$ is a key polynomial for $\mu$.  Take $f,g\in \K[x]$ with $\deg(f)<\deg(Q)$ and $\deg(g)<\deg(Q)$ and suppose $fg = lQ+r$ is the $Q$-expansion of $fg$. We will prove that $\mu(fg) = \mu(r)<\mu(lQ)$. 
 Since $\deg(f),\deg(g),\deg(l),\deg(r)<\deg(Q)$, by the minimality of $\deg(Q)$ all $f,g,l$ and $r$ are $\mathfrak{v}$-stable. Since  $\mu(Q)=\gamma>\nu_i(Q)$ for every $i\in I$, we take $k>\max\{i_{f}, i_{g}, i_{l}, i_{r}\}$ and then
 \begin{align*}\mu(lQ)&=\mu(l)+\mu(Q)\\
 &  > \nu_k(l)+ \nu_k(Q)\\
 & \geq \min\{\nu_k(r), \nu_k(fg) \} \\
 & = \min\{\mu(r), \mu(fg) \}.  
 \end{align*}
By Proposition~\ref{propequivPolichave}, $Q$ is a key polynomial for $\mu$. 
%
\end{enumerate}
\end{proof}

Let $R_Q$ be the additive subgroup of $\mathcal{G}_\mu$ generated by the set $\{ \inv_\mu(f)\mid f\in \K[x]_d\}$. Since $Q$ is a key polynomial for $\mu$, Proposition~\ref{propequivPolichaveRq} guarantees that $R_Q$ is a subring of $\mathcal{G}_\mu$. 

\vspace{0.2cm}

\begin{Cor}\label{teoLimitSeqNotStable} Let $\mathfrak{v}= \{\nu_i\}_{i\in I}$ be a totally ordered set in $\mathcal{V}$ such that $\mathfrak{v}$ has no maximum and there is at least one polynomial that is not $\mathfrak{v}$-stable. 
Let $Q$ be a monic polynomial of smallest degree $d$ that is not $\mathfrak{v}$-stable and take $\gamma\in \Gamma_\infty$ such that $\gamma>\nu_i(Q)$ for every $i\in I$. Take $\mu$ as in Proposition~\ref{lemamuQgammaVal} and $R_Q$  as in the above paragraph.   Consider the direct system $\{( \mathcal{G}_{\nu_i}, \phi_{ij})\}^{i,j\in I }_{i\leq j}$. Then $\underset{\longrightarrow}{\lim\,} \mathcal{G}_{\nu_i}\cong R_{Q}$  
as commutative rings with unity.

\end{Cor}

\begin{proof} By  Theorem~\ref{teoLimVmaiorRstable} and Lemma~\ref{lemR}, we have $\underset{\longrightarrow}{\lim\,}\mathcal{G}_{\nu_i}\cong R = R_Q$.
\end{proof}

\begin{Cor}\label{corLimY}We have 
$$\mathcal{G}_\mu \cong (\underset{\longrightarrow}{\lim\,}\mathcal{G}_{\nu_i})[Y],  $$
where $Y$ is an indeterminate over $\underset{\longrightarrow}{\lim\,}\mathcal{G}_{\nu_i}$.

\end{Cor}

\begin{proof}
By Proposition~\ref{lemAlgGraduadaAnelPoliy}, $\mathcal{G}_\mu = R_Q[y_Q]$, with $y_Q =\inv_\mu(Q)$. Then, we extend the isomorphism  $\underset{\longrightarrow}{\lim\,}\mathcal{G}_{\nu_i}\cong R_Q$ to an isomorphism  $ ( \underset{\longrightarrow}{\lim\,}\mathcal{G}_{\nu_i})[Y] \cong \mathcal{G}_\mu  $ by sending $Y$ to $y_Q$.

\end{proof}

\section{Limit key polynomials and direct limits}\label{LKPLimit}

%
%

Let $\nu$ be a valuation on $\K[x]$ extending $\nu_0$. 
%
%
%
%
%
%
%
%
%
Take $n\in \N$ and let $\Psi_n$ denote the set of all key polynomials for $\nu$ with degree $n$. Assume that $\Psi_n\neq \emptyset$ and that  $\{ \nu(Q)\mid Q\in \Psi_n \}$ does not have a maximum. 
Consider the set 
$$\mathcal{K}_n:=\{ f\in \K[x]\mid  \nu_Q(f)<\nu(f) \text{ for all } Q\in \Psi_n \}.$$


\begin{Lema} Take  $Q_n\in \mathcal{K}_n$ a monic of polynomial of least degree. Then $Q_n$ is a key polynomial for $\nu$. 

\end{Lema}

\begin{proof}Consider $f_1,f_2\in \K[x]$ with $\deg(f_1),\deg(f_2)<\deg(Q_n)$ and suppose $f_1f_2 = f_3Q_n+f_4$ is the $Q_n$ expansion of $f_1f_2$. Since each $f_l$, $1\leq l\leq 4$, has degree smaller than $\deg(Q_n)$, it follows by the minimality of $\deg(Q_n)$ that $f_l\not\in  \mathcal{K}_n$. Hence, for each $l$, $1\leq l\leq 4$, there exists $Q_l\in \Psi_n$ such that $\nu_{Q_l}(f_l)=\nu(f_l)$. Since $\{ \nu(Q)\mid Q\in \Psi_n \}$ does not have a maximum, there exists $Q\in \Psi_n$ such that $\nu(Q_l)<\nu(Q)$ for all $l$, $1\leq l\leq 4$. By Lemma~\ref{prop3itensQQlinhaPoliChaves} (iii) and Lemma~\ref{lempolichaveEpsilonmenor} (iii), it follows that $\nu_Q(f_l) = \nu(f_l)$ for every $l$, $1\leq l\leq 4$. Thus,
\begin{align*}\nu(f_3Q_n) &= \nu(f_3)+\nu(Q_n)\\
& >\nu_Q(f_3)+\nu_Q(Q_n)\\
&=\nu_Q(f_1f_2-f_4)\\
&\geq \min\{\nu_Q(f_1f_2),\nu_Q(f_4) \}\\
&=  \min\{\nu(f_1f_2),\nu(f_4) \}.
\end{align*}

Therefore, $\nu(f_1f_2)=\nu(f_4)<\nu(f_3Q_n)$ and then $Q_n$ is a key polynomial for $\nu$ by Proposition~\ref{propequivPolichave}. 

\end{proof}

\begin{Def}\label{defLKP2}
	 A monic polynomial $Q_n\in \K[x]$ is called a \textbf{limit key polynomial}  for $\Psi_n$
	  if $Q_n\in \mathcal{K}_n$ and $Q_n$ has the least degree among polynomials in $\mathcal{K}_n$. 
\end{Def}

\vspace{0.2cm}

Consider the following relation on $\Psi_n$:
\begin{equation}\label{QmenorQlinha}
Q\preceq Q' \Leftrightarrow \nu_Q\leq \nu_{Q'} \text{ and } Q\prec Q' \Leftrightarrow \nu_Q<\nu_{Q'}.
\end{equation}
	
	We note that if we take $Q,Q'\in \Psi_n$, then either $Q\preceq Q'$ or $Q'\preceq Q$. Indeed, considering $\delta(Q), \delta(Q')\in \Gamma_\Q$, since they belong to a totally ordered group, we have $\delta(Q)\leq \delta(Q')$ or $\delta(Q')\leq  \delta(Q)$. By Lemma~\ref{lempolichaveEpsilonmenor} (2), for every $f\in \K[x]$ we have $\nu_{Q}(f)\leq \nu_{Q'}(f)$ or $\nu_{Q'}(f)\leq \nu_{Q}(f)$, that is, $Q\preceq Q'$ or $Q'\preceq Q$.
	Therefore, with this pre-order
	 $(\Psi_n, \preceq)$ is a directed set. It follows also that $\mathfrak{v} = \{\nu_Q  \}_{Q\in \Psi_n}$ is a totally ordered set.

\begin{Cor}
	Consider the family of graded rings $\{  \mathcal{G}_Q\}_{Q\in \Psi_n}$ and, for $Q \preceq Q'$, let $\phi_{QQ'}$ be the map
	\begin{align*}
	\phi_{QQ'}: \hspace{0.5cm} \mathcal G_Q\hspace{0.3cm}   &\longrightarrow \mathcal G_{Q'}\\
	\inv_Q(f) &\longmapsto  \begin{cases}
		\inv_{Q'}(f)& \mbox{ if }\nu_Q(f)=\nu_{Q'}(f)\\ 
		0&\mbox{ if }\nu_Q(f)<\nu_{Q'}(f),
	\end{cases} 
	\end{align*}
extended naturally to arbitrary (that is, not necessarily homogeneous) elements of $\mathcal G_Q$. Then  $\{( \mathcal{G}_Q, \phi_{QQ'})\}^{Q,Q'\in \Psi_n}_{Q\preceq Q'}$ is a direct system over $\Psi_n$.

\end{Cor}

\begin{proof} This follows from Lemma~\ref{lemDirectSystemValTotalOrder} because $\{  \nu_Q\}_{Q\in \Psi_n}$ is a totally ordered set. 
%
%
%
%
%
%

\end{proof}



%

%
%
%
%
%
%
%
%
%
%

In the next lemma we gather some properties of limit key polynomials for $\Psi_n$. These properties will allow us to prove that  $\mathfrak{v} = \{\nu_Q  \}_{Q\in \Psi_n}$
has no maximum.

\vspace{0.2cm} 

\begin{Lema}\label{lemPropQn}Let $Q_n$ be a limit key polynomial for $\Psi_n$. 

\begin{enumerate}
\item We have $\delta(Q)<\delta(Q_n)$ for every $Q\in \Psi_n$. Hence, $\deg(Q_n)\geq n$. 

\vspace{0.1cm} 

\item For every $Q\in \Psi_n$, we have $\nu_Q(f)\leq \nu_{Q_n}(f)$ for all $f\in \K[x]$ and \linebreak $\nu_Q(Q_n)<\nu_{Q_n}(Q_n)=\nu(Q_n)$. Also,  $\nu_{Q}(Q_n)<\nu_{Q'}(Q_n)$ for every $Q\prec Q'$ in $\Psi_n$.

\item If $\deg(f)< \deg(Q_n)$, then there exists $\bar{Q}\in \Psi_n$ (depending on $f$) such that  $\nu_{\bar{Q}}(f)=\nu_{Q_n}(f)=\nu(f)$. Moreover, $\nu_{Q}(f)=\nu_{\bar{Q}}(f)$ for every $Q\in \Psi_n$ with $\bar{Q} \preceq Q$. 

%

\end{enumerate}

\end{Lema}


\begin{proof}
\noindent \begin{enumerate}
\item Suppose $\delta(Q_n)\leq \delta(Q)$. Hence, by Lemma~\ref{lempolichaveEpsilonmenor} (2), $\nu_{Q_n}(f)\leq \nu_Q(f)$ for every $f\in \K[x]$. In particular, $\nu(Q_n)=\nu_{Q_n}(Q_n)\leq \nu_Q(Q_n)$. However, this contradicts $Q_n\in \mathcal{K}_n$. Therefore, $\delta(Q)<\delta(Q_n)$ for every $Q\in \Psi_n$. Now suppose $\deg(Q_n)<n=\deg(Q)$. By Lemma~\ref{prop3itensQQlinhaPoliChaves} (1), we would have $\delta(Q_n)<\delta(Q)$, a contradiction. Thus, $\deg(Q_n)\geq n$.

\vspace{0.1cm} 

\item By Lemma~\ref{lempolichaveEpsilonmenor}, Lemma~\ref{prop3itensQQlinhaPoliChaves} and the preceding item, we have \linebreak $\nu_Q(f)\leq \nu_{Q_n}(f)$ for all $f\in \K[x]$ and $\nu_Q(Q_n)<\nu_{Q_n}(Q_n)=\nu(Q_n)$ for every $Q\in \Psi_n$. 
Also, take $Q\prec Q'$, that is, there exists $g\in \K[x]$ such that $\nu_{Q}(g)<\nu_{Q'}(g)$. Thus, by Lemma~\ref{lempolichaveEpsilonmenor} (2), we must have $\delta(Q)<\delta(Q')$.  Since  $\nu_{Q'}(Q_n)<\nu(Q_n)$, we see by Lemma~\ref{lempolichaveEpsilonmenor} (4) that $\nu_{Q}(Q_n)<\nu_{Q'}(Q_n)$.  

\vspace{0.2cm} 

\item Take $f\in \K[x]$ such that $\deg(f)<\deg(Q_n)$. By the preceding item, \linebreak $\nu_Q(f)\leq \nu_{Q_n}(f)=\nu(f)$ for every $Q\in \Psi_n$.  Since $Q_n$ has the minimal degree among polynomials in $\mathcal{K}_n$, there exists $\bar{Q}\in \Psi_n$ such that $\nu_{\bar{Q}}(f)= \nu_{Q_n}(f)=\nu(f)$. Now, if $Q\in \Psi_n$ is such that $\bar{Q}\preceq Q$, then 
$$\nu_{Q_n}(f)\geq \nu_Q(f)\geq\nu_{\bar{Q}}(f) = \nu_{Q_n}(f).  $$
Hence, $\nu_Q(f)=\nu_{\bar{Q}}(f)$. 

\end{enumerate}

\end{proof}

\begin{Cor}\label{corQnNotvStable} 
The family $\mathfrak{v} = \{\nu_Q  \}_{Q\in \Psi_n}$ is a totally ordered subset of $\mathcal{V}$ with no maximum. Also, $Q_n$ is a polynomial of least degree that is not $\mathfrak{v}$-stable. If we take $\gamma=\nu(Q_n)$, then $\nu_{Q_n}(f)=\underset{0\leq j \leq r}{\min}\{ \mathfrak{v}(f_j)+j\gamma  \}$, where $f_0, \ldots, f_r$ are the coefficients of the $Q_n$-expansion of $f$.
\end{Cor}

\begin{proof} 

\vspace{0.3cm}

We already saw that $\mathfrak{v} = \{\nu_Q  \}_{Q\in \Psi_n}$ is a totally ordered set. Suppose, aiming for a contradiction, that $\mathfrak{v} $ has maximum $\nu_{Q''}$. Then $\nu_{Q}(f)\leq \nu_{Q''}(f)$ for all $Q\in \Psi_n$ and $f\in \K[x]$. Since $\{ \nu(Q)\mid Q\in \Psi_n \}$  does not have a  maximum, there exists $Q'\in \Psi_n$ such that $\delta(Q'')<\delta(Q')$ (Lemma~\ref{prop3itensQQlinhaPoliChaves}). By Lemma~\ref{lempolichaveEpsilonmenor} (2), $\nu_{Q''}(f)\leq \nu_{Q'}(f)$, which implies $\nu_{Q''}(f)=\nu_{Q'}(f)$ for all $f\in \K[x]$. However, by Lemma~\ref{lemPropQn} (2), we have $\nu_{Q'}(Q_n)<\nu(Q_n)$ and this, together with Lemma~\ref{lempolichaveEpsilonmenor} (4), implies $\nu_{Q''}(Q_n)<\nu_{Q'}(Q_n)$, a contradiction to the maximality of $\nu_{Q''}$. Therefore, $\mathfrak{v}$ has no maximum.

\vspace{0.2cm} 	

By Lemma~\ref{lemPropQn} (2), we have $\nu_{Q}(Q_n)<\nu_{Q'}(Q_n)$ for $Q\prec Q'$, that is, $Q_n$ is not $\mathfrak{v}$-stable. Also,  we have $\nu_{Q}(Q_n)<\nu(Q_n)=\gamma$ for all $Q\in \Psi_n$.
Moreover, by Lemma~\ref{lemPropQn} (3), if $\deg(g)<\deg(Q_n)$, then there exists $\bar{Q}\in \Psi_n$ such that $\nu_{\bar{Q}}(g)= \nu_{Q}(g) $ for every $Q\in \Psi_n$ with $\bar{Q}\preceq Q$. That is, $g$ is $\mathfrak{v}$-stable and  $\nu(g) = \nu_{\bar{Q}}(g)=\mathfrak{v}(g)$. 
%
Thus, for every $f\in \K[x]$, we can write $f = f_0+f_1Q_n+\ldots+f_rQ_{n}^r$ and conclude that
$$\nu_{Q_n}(f) = \underset{0\leq j \leq r}{\min}\{ \nu(f_j)+j\nu(Q_n)  \} = \underset{0\leq j \leq r}{\min}\{ \mathfrak{v}(f_j)+j\gamma  \}.   $$
\end{proof}

\begin{Cor}\label{corLimQn} We have $\underset{\longrightarrow}{\lim\,}\mathcal{G}_Q\cong R_{Q_n}$ as commutative rings with unity and $\mathcal{G}_{Q_n}\cong (\underset{\longrightarrow}{\lim\,}\mathcal{G}_Q)[Y]$.

\end{Cor}

\begin{proof} 
It follows from Corollary~\ref{teoLimitSeqNotStable} and Corollary~\ref{corLimY}. 

\end{proof}

\section{Valuation-algebraic valuations and direct limits}\label{ValAlgLimit}

In this last section, we give an application for Corollary~\ref{teoLimitSeqStable}. We start defining the concepts of valuation-transcendental and valuation-algebraic valuations.

\vspace{0.2cm} 

\begin{Def} A valuation $\nu$ on $\K[x]$  extending $\nu_0$ is called \linebreak \textbf{value-transcendental} if either it is not Krull or the quotient group $\nu(\K[x])/\nu_0\K$ is not a torsion group. We say that $\nu$ is \textbf{residue-transcendental} if it is Krull and the field extension $\K(x)\nu\mid \K\nu_0$ is transcendental, where we denote also by $\nu$ the unique extension of $\nu$ from $\K[x]$ to $\K(x)$.
\end{Def}

\vspace{0.1cm} 

\begin{Def}\label{defValAlg} A valuation $\nu$ on $\K[x]$  extending $\nu_0$  is called \textbf{valuation-transcen-\\dental} if it is value-transcendental or residue-transcendental. We say that $\nu$ is \textbf{valuation-algebraic} if it is not valuation-transcendental.
 
\end{Def}

\begin{Obs}\label{obsAbhyankar} By Abhyankar's inequality (see \cite{zar}, p.330), we see that a valuation cannot be  value-transcendental and residue-transcendental at the same time.

\end{Obs}

\vspace{0.1cm} 

\begin{Obs}Explicitly, a valuation $\nu$ on $\K[x]$  extending $\nu_0$  is valuation-algebraic if it is a Krull valuation, $\nu(\K[x])/\nu_0\K$ is a torsion group and $\K(x)\nu\mid \K\nu_0$ is an algebraic field extension.

\end{Obs}

\vspace{0.1cm}

\begin{Lema}\label{lemValuationAlgTrunc} Let $\nu$ be a valuation-algebraic valuation on $\K[x]$  extending $\nu_0$ on $\K$. Suppose that $q$ is a polynomial such that $\nu_q$ is a valuation. Then $\nu_q$ is residue-transcendental.
\end{Lema}

\begin{proof} By Theorem 3.1 of \cite{josneiKeyPolyMinimalPairs}, we have that $\nu_q$ is valuation-transcendental. Given $f\in \K[x]$, $f\neq 0$, we know that $$\nu_q(f) = \min_{0\leq i \leq r} \{ \nu(f_i)+i\nu(q)  \}\in \nu(\K[x]),$$ where $f_0, \ldots, f_r$ are the coefficients  of the $q$-expansion of $f$. Since $\nu$ is valuation-algebraic, $\nu(\K[x])/\nu_0\K$ is a torsion group. Hence, $\nu_q(f)$ is a torsion element in $\nu_q(\K[x])/\nu_0\K$ for every $f\in \K[x]$, $f\neq 0$. Therefore, $\nu_q$ is not value-transcendental and, due to Remark~\ref{obsAbhyankar}, we see that $\nu_q$ is a residue-transcendental valuation.

\end{proof}

To a given valuation-algebraic valuation $\nu$, we are going to associate a totally ordered subset of $\mathcal{V}$. In order to do that, we use the results of \cite{josneiKeyPolyPropriedades} on complete sets, which we define bellow. 

\vspace{0.2cm} 

\begin{Def} Let $\nu$ be a valuation on $\K[x]$. A set $\textbf{Q}\subset \K[x]$ of key polynomials for $\nu$ is called a \textbf{complete set of key polynomials} for $\nu$ if for every $f\in \K[x]$ with $\deg(f)\geq 1$ there exists $Q\in \textbf{Q}$ with $\deg(Q)\leq \deg(f)$ such that $\nu_Q(f)=\nu(f)$.
\end{Def}

\vspace{0.1cm}

\begin{Prop}\label{propcompleteset}(Theorem 1.1 of \cite{josneiKeyPolyPropriedades})
Every valuation $\nu$ on $\K[x]$ admits a complete set \textbf{Q} of key polynomials.
\end{Prop}

\vspace{0.1cm} 

\begin{Obs} As remarked in \cite{josneimonomial}, the definition of complete set in Theorem 1.1 of \cite{josneiKeyPolyPropriedades}  does not require that $\deg(Q)\leq \deg(f)$. However, the proof of the Theorem shows that this inequality always holds.
\end{Obs}
%
\vspace{0.1cm} 

\begin{Prop}\label{propValAlgStable} Let $\nu\in \mathcal{V}$ be a valuation-algebraic valuation. Then there exists a totally ordered subset  $\mathfrak{v} = \{ \nu_i \}_{i\in I}\subset \mathcal{V}$  without maximum such that every $f\in \K[x]$ is $\mathfrak{v}$-stable, $\nu = \displaystyle \sup_{i\in I}\nu_i$ and each $\nu_i$ is residue-transcendental.
\end{Prop}

\begin{proof}
By Proposition~\ref{propcompleteset}, there exists a complete set \textbf{Q} of key polynomials for $\nu$. Consider $\mathfrak{v} = \{ \nu_{Q} \}_{Q\in \textbf{Q}}\subset \mathcal{V}$, which is totally ordered due to Lemma~\ref{lempolichaveEpsilonmenor} (2). 
We order the set $\textbf{Q}$ by posing $Q\preceq Q'$ if and only if $\nu_Q\leq \nu_{Q'}$. By Lemma~\ref{lemValuationAlgTrunc}, each $\nu_Q$ is residue-transcendental.

\vspace{0.2cm} 

We now show that every $f\in \K[x]$ is $\mathfrak{v}$-stable. Indeed, for every $f\in \K[x]$, there exists $Q\in \textbf{Q}$ such that $\deg(Q)\leq \deg(f)$ and $\nu_Q(f)=\nu(f)$. If $Q\preceq Q'$, then we have the following:
\begin{itemize}
\item if $\delta(Q')\leq \delta(Q)$, then $\nu_{Q'}(f)\leq \nu_Q(f)$ (Lemma~\ref{lempolichaveEpsilonmenor} (2)), that is, $\nu_{Q'}(f) = \nu_Q(f)=\nu(f)$;

\item if $\delta(Q)< \delta(Q')$, then by Lemma~\ref{lempolichaveEpsilonmenor} (3) we have that $\nu_Q(f)=\nu(f)$ implies $\nu_{Q'}(f)=\nu(f)$.
\end{itemize}
Hence, $Q\preceq Q'$ implies $\nu_{Q'}(f)=\nu_{Q}(f)=\nu(f)$. That is, $f$ is $\mathfrak{v}$-stable and $\nu=\underset{Q\in \textbf{Q}}{\sup} \nu_Q$ as in  Proposition~\ref{lemVstable}. 
Moreover, suppose  $\{ \nu_{Q} \}_{Q\in \textbf{Q}}$ has a maximum. Then $$\nu =\underset{Q\in \textbf{Q}}{\sup} \nu_Q = \nu_{Q_m}$$ for some $Q_m\in \textbf{Q}$, which is a contradiction since $\nu$ is valuation-algebraic and $\nu_{Q_m}$ is residue-transcendental. Therefore, $\{ \nu_{Q} \}_{Q\in \textbf{Q}}$ does not have a maximum.

\end{proof}

\begin{Cor}\label{corValgLim}
Let $\nu\in \mathcal{V}$ be a valuation-algebraic valuation and take 
$\mathfrak{v} = \{ \nu_{i} \}_{i\in I}\subset \mathcal{V}$ as in Proposition~\ref{propValAlgStable}. Then $\underset{\longrightarrow}{\lim\,}\mathcal{G}_{\nu_i}\cong \mathcal{G}_{\nu}$  
as commutative rings with unity.
\end{Cor}

\begin{proof}
It follows from Corollary~\ref{teoLimitSeqStable}.

\end{proof}

\end{document}